\documentclass[11pt, reqno]{amsart}

\usepackage{algorithmic}
\usepackage[dvipsnames]{xcolor}
\usepackage{tikz}
\usepackage{stmaryrd}
\usetikzlibrary{matrix,arrows,decorations.pathmorphing}
\usepackage{amsfonts}
\usepackage{pdfpages}
\usepackage{graphicx}
\usepackage{mathpazo}
\usepackage{euler}
\usepackage{amssymb}
\usepackage{amsmath,mathtools}
\usepackage{fullpage}
\usepackage{caption}
\usepackage{amsthm}
\usepackage{thmtools}
\usepackage{thm-restate}
\usepackage{enumerate}
\usepackage{picinpar} % for pictures in paragraphs
\usepackage{tikz-cd}
\usepackage{color}
\usepackage{latexsym}
\usepackage[toc,page]{appendix}
\usepackage{url} 
\usepackage{multicol}
\usepackage{rotating}
\usepackage[numbers]{natbib}         % bibliography style
\usepackage[colorlinks]{hyperref} 
\hypersetup{colorlinks=true, linkcolor=cyan, citecolor = OliveGreen, urlcolor = black}
\makeatletter
\newcommand{\customlabel}[2]{%
   \protected@write \@auxout {}{\string \newlabel {#1}{{#2}{\thepage}{#2}{#1}{}} }%
   \hypertarget{#1}{#2}}
\makeatother
\usepackage{setspace}
\usepackage{mathrsfs}      % for \mathscr

\usepackage{float}
\usepackage{setspace}
\usepackage{indentfirst}           % to indent the first paragraph of a section
\usepackage{listings}
\usepackage{comment}
\usepackage{manfnt}
\usepackage{longtable}
\usepackage{cleveref}
\makeatletter
\def\@tocline#1#2#3#4#5#6#7{\relax
  \ifnum #1>\c@tocdepth % then omit
  \else
    \par \addpenalty\@secpenalty\addvspace{#2}%
    \begingroup \hyphenpenalty\@M
    \@ifempty{#4}{%
      \@tempdima\csname r@tocindent\number#1\endcsname\relax
    }{%
      \@tempdima#4\relax
    }%
    \parindent\z@ \leftskip#3\relax \advance\leftskip\@tempdima\relax
    \rightskip\@pnumwidth plus4em \parfillskip-\@pnumwidth
    #5\leavevmode\hskip-\@tempdima
      \ifcase #1
       \or\or \hskip 1em \or \hskip 2em \else \hskip 3em \fi%
      #6\nobreak\relax
    \dotfill\hbox to\@pnumwidth{\@tocpagenum{#7}}\par
    \nobreak
    \endgroup
  \fi}
\makeatother

\makeatletter
\def\@tocline#1#2#3#4#5#6#7{\relax
  \ifnum #1>\c@tocdepth % then omit
  \else
    \par \addpenalty\@secpenalty\addvspace{#2}%
    \begingroup \hyphenpenalty\@M
    \@ifempty{#4}{%
      \@tempdima\csname r@tocindent\number#1\endcsname\relax
    }{%
      \@tempdima#4\relax
    }%
    \parindent\z@ \leftskip#3\relax \advance\leftskip\@tempdima\relax
    \rightskip\@pnumwidth plus4em \parfillskip-\@pnumwidth
    #5\leavevmode\hskip-\@tempdima
      \ifcase #1
       \or\or \hskip 1em \or \hskip 2em \else \hskip 3em \fi%
      #6\nobreak\relax
    \dotfill\hbox to\@pnumwidth{\@tocpagenum{#7}}\par
    \nobreak
    \endgroup
  \fi}
\makeatother

\numberwithin{equation}{subsection}

\numberwithin{equation}{subsection}

\newtheorem{theorem}[subsection]{Theorem}
\newtheorem{lemma}[subsection]{Lemma}
\newtheorem{coro}[subsection]{Corollary}

\newtheorem{prop}[subsection]{Proposition}

\theoremstyle{definition}
\newtheorem{defn}[subsection]{Definition}
\newtheorem{algorithm}[subsection]{Algorithm}
\newtheorem{subroutine}[subsection]{Subroutine}
\newtheorem{exam}[subsection]{Example}

\newtheorem{remark}[subsection]{Remark}

\theoremstyle{remark}

\newcommand{\Cl}[1]{\text{Cl}#1}

\newcommand{\ZZ}{\mathbb Z}

\def\quotient#1#2{\raise1ex\hbox{$#1$}{\Large/} \lower1ex\hbox{$#2$}}

\newcommand{\halp}[1]{\textcolor{red}{#1}}

\DeclareMathOperator{\Frac}{Frac}

\DeclareMathOperator{\Hom}{Hom}
\DeclareMathOperator{\rank}{rank}

\DeclareMathOperator{\tors}{tors}
\DeclareMathOperator{\img}{Im}

\DeclareMathOperator{\Spec}{Spec}

\DeclareMathOperator{\Div}{Div}

\DeclareMathOperator{\divs}{div}

\DeclareMathOperator{\Trop}{Trop}

\DeclareMathOperator{\Proj}{Proj}

\usetikzlibrary{positioning}

\begin{document}

%\begin{titlepage}
%\vspace*{\fill}
\title{Computing Unit Groups of Curves}

\author{Justin Chen}
\email{jchen@math.berkeley.edu}
\author{Sameera Vemulapalli}
\email{sameerav@princeton.edu}
\author{Leon Zhang}
\email{leonyz@math.berkeley.edu}

\begin{abstract}
  The group of units modulo constants of an affine variety over an algebraically closed field is free abelian of finite rank. Computing this group is difficult but of fundamental importance in tropical geometry, where it is desirable to realize intrinsic tropicalizations. We present practical algorithms for computing unit groups of smooth curves of low genus. Our approach is rooted in divisor theory, based on interpolation in the case of rational curves and on methods from algebraic number theory in the case of elliptic curves.
  \end{abstract}
\date{\today}
\maketitle

\section{Introduction}

Among the invariants of a commutative ring, the group of units is one of the most fundamental. However, explicit computation of this group is difficult, and even its structure remains mysterious in general \cite{Fuchs}. To date, most progress has centered on rings of integers of algebraic number fields, or localizations thereof, driven by a need for practical algorithms in computational number theory \cite{cohenbook}. These results rely fundamentally on Dirichlet's unit theorem, which describes the group of units, modulo torsion, of a number field as a free abelian group of finite rank specified by simple invariants of the number field.

An analogous theorem of Samuel \cite{samuel} states that for a finitely generated domain over an algebraically closed field, the group of units, modulo scalars, is free abelian of finite rank. In contrast to the number field case, no formula for the rank is known. Given the coordinate ring of a very affine variety, a basis for its unit group yields an embedding of the variety into its so-called \emph{intrinsic torus} \cite{tropicalbook}. In tropical geometry, this embedding of a very affine variety into its intrinsic torus realizes its \emph{intrinsic tropicalization}, from which all other tropicalizations can be recovered. However, explicitly computing the intrinsic tropicalization is difficult, because one must first compute the unit group.

In this work we describe effective methods for computing unit groups of smooth very affine curves of low genus. Our methods rely on divisor theory for projective varieties: we embed the unit group of a very affine variety into the Weil divisor group of the projective closure, and study the cokernel of this embedding as a subgroup of the divisor class group. This allows us to give algorithms for computing unit groups of rational normal curves and elliptic curves:

\begin{restatable}{thmm}{rationalnormalcurvesoln}
\label{rationalnormalcurvesoln}
Let $\overline{C} \subseteq \mathbb{P}^n_k$ be a rational normal curve over an algebraically closed field $k$, given parametrically as the image of a map $\mathbb{P}^1_k \hookrightarrow \mathbb{P}^n_k$. Let $C := \overline{C} \cap \mathbb{T}^n$ be the corresponding very affine curve, with coordinate ring $R$. Then \Cref{rationalcurvealg} correctly computes a $\mathbb Z$-basis of $R^*/k^*$.
\end{restatable}
\begin{restatable}{thmm}{ellipticcurveqonesoln}
\label{ellipticcurveqonesoln}
Let $k = \overline{\mathbb{Q}}$, let $\overline{E} \subseteq \mathbb{P}^2_k$ be an elliptic curve, and let $E := \overline{E} \cap \mathbb{T}^2$ be the corresponding very affine elliptic curve with coordinate ring $R$. Then \Cref{ellcurvealg} correctly computes a $\mathbb Z$-basis of $R^*/k^*$.
\end{restatable}

We briefly describe the structure of the paper. The basics of Samuel's theorem and intrinsic tropicalizations are discussed in \Cref{problem}. In \Cref{preliminaryresults} we develop the relationship between our problem and the geometry of boundary divisors, and describe a simple algorithm for interpolating divisors of rational functions in terms of Laurent polynomials, when possible. We consider the families of Fermat curves and plane conics in \Cref{planequadrics}, and rational normal curves in parametric form in \Cref{rationalnormalcurves}. Finally, we discuss elliptic curves in \Cref{ellipticcurves}.

Many of our algorithms have been implemented in Macaulay2 \cite{M2}, Singular \cite{DGPS}, and Sage \cite{sagemath}. Our code for the examples in this paper can be found at our supplementary materials website:

\begin{center}{\url{https://math.berkeley.edu/~leonyz/code/}}
\end{center}

\subsection{Acknowledgements}
Leon Zhang and Sameera Vemulapalli would like to thank the Max Planck Institute for Mathematics in the Sciences for its hospitality while working on this project. Leon Zhang was supported by a National Science Foundation Graduate Research Fellowship.

The authors thank Bernd Sturmfels for suggesting and advising this project. We would also like to thank Chris Eur and Martin Helmer for helpful discussions, Yue Ren for generous and thoughtful help in computing tropicalizations in Singular \cite{DGPS}, and Bjorn Poonen and Ronald van Luijk for their expertise and guidance.

\section{Background}

\label{problem}

We begin by stating the problem in a general setting. Let $k$ be an algebraically closed field, and let $R$ be a finitely generated $k$-algebra which is a domain. The inclusion $k \subseteq R$ induces a short exact sequence of multiplicative abelian groups
\begin{equation} \label{eq:1}
1 \xrightarrow{\qquad} k^* \xrightarrow{\qquad} R^* \xrightarrow{\qquad} R^*/k^* \xrightarrow{\qquad} 1
\end{equation}

Our goal is to compute, as explicitly as possible, the group $R^*/k^*$. Although this may seem to be a purely algebraic problem, the key to progress is to use insights from geometry, particularly divisor theory on projective varieties. Thus, writing $R = k[x_1, \ldots, x_n]/I$ as a quotient of a polynomial ring by a prime ideal $I$, set $X := \Spec R \subseteq \mathbb{A}^n_k$, the affine variety corresponding to $R$, and let $\overline{X} \subseteq \mathbb{P}^n_k$ denote the projective closure of $X$ in projective $n$-space. Write $\partial X := \overline{X} \setminus X = \overline{X} \cap V(x_0)$ for the boundary of $\overline{X}$, which is the intersection of $\overline{X}$ with the hyperplane at infinity in $\mathbb{P}^n_k$.

The main point is that a unit in $R$ corresponds, via homogenization, exactly to a rational function on $\overline{X}$ which has zeros and poles only on $\partial X$. To be precise:

\begin{lemma} \label{homogenizationLemma}
With notation as above, let $R \to \overline{R}$ be the homogenization map $f \mapsto \overline{f} := x_0^{\deg f}f (\frac{x_i}{x_0})$. Then:

i) For any $f, g \in R$, $\overline{fg} = \overline{f} \overline{g}$, and

ii) $f \in R^*$ if and only if $V(\overline{f}) \cap \overline{X} \subseteq \partial X$. 
\end{lemma}

\begin{proof} 
First, note that dehomogenization is evaluation at $x_0 = 1$, hence is a ring map with kernel $(x_0 - 1)$. As the kernel contains no nonzero homogeneous elements, it follows that if $f_1, f_2$ are homogeneous of the same degree with the same dehomogenization, then $f_1 = f_2$. 

i) Since $\overline{fg}$ and $\overline{f}\overline{g}$ are both homogeneous of the same degree and dehomogenize to $fg$, by the above reasoning they must be equal.

ii) Recall that $\partial X=\overline X \cap V(x_0)$, so it suffices to show that $f \in R^*$ if and only if $V(\overline{f}) \cap \overline{X} \subseteq V(x_0)$. If $g_1, \ldots, g_r$ is a Gr\"obner basis for the defining ideal $I$ of $X$, then $\overline{X}$ has defining ideal $(\overline{g_1}, \ldots, \overline{g_r})$ \cite[Prop. 15.31]{Eisenbud}. It thus suffices to show $1 \in (f, g_1, \ldots, g_r)$ if and only if $x_0 \in \sqrt{(\overline{f}, \overline{g_1}, \ldots, \overline{g_r})}$. The ``if" direction follows by dehomogenizing. For the ``only if'' direction, pick $h$ with $1 - fh \in I$. Then $\overline{1 - fh} \in (\overline{g_1}, \ldots, \overline{g_r})$. But $\overline{1 - fh} = x_0^d - \overline{fh}$, where $d = \deg(fh)$, as both sides are homogeneous and dehomogenize to $1 - fh$. By (i) therefore, $x_0^d \in (\overline{f}, \overline{g_1}, \ldots, \overline{g_r})$ as desired.
\end{proof}

Suppose now that $\overline{X}$ is normal, and write $\Div(\overline{X})$ (resp. $\Cl(\overline{X})$) for the group of Weil divisors (resp. the divisor class group) on $\overline{X}$. Let $\Div^0(\overline{X})$ (resp. $\Cl^0(\overline{X})$) denote the subgroup of divisors (resp. divisor classes) of degree zero.

\begin{defn} \label{div0Boundary}
We define 
\[
\Div_{\partial}^0(\overline{X}) := \left\{ \sum_{\textup{finite}} a_i P_i \; \Big| \; P_i {\text{ component of }} \partial X, a_i \in \mathbb{Z}, \sum a_i = 0 \right\} \subseteq \Div^0(\overline{X})
\]
i.e. the subgroup of $\Div^0(\overline{X})$ supported on $\partial X$. This makes sense since $\partial X$ has codimension 1 in $\overline X$.
\end{defn}

Now, homogenization gives a natural map $R^* \hookrightarrow \Frac(\overline{R})^*$, which is a homomorphism of multiplicative groups by \Cref{homogenizationLemma}(i). Composing with the natural map $\Frac(\overline{R})^* \to \Div^0(\overline{X})$, $f \mapsto \divs(f)$ gives a homomorphism $\widetilde{\phi} : R^* \to \Div^0(\overline{X})$ from a multiplicative abelian group to an additive abelian group. Since a unit is a rational function which is invertible on $X$, hence has zeros and poles only on $\partial X$ by \Cref{homogenizationLemma}(ii), this shows that the image of $\widetilde{\phi}$ is contained in $\Div^0_\partial(\overline{X})$. Next, the kernel of $\widetilde{\phi}$ consists of units whose associated rational function has no zeros or poles anywhere on $\overline{X}$. Such an element must be a scalar, i.e. comes from $k^*$, so we have an induced map $\phi : R^*/k^* \hookrightarrow \Div_{\partial}^0(\overline{X})$.

Putting the above reasoning together yields a classical theorem of Samuel \protect{\cite{samuel}} on the structure of the unit group: 

\begin{theorem}[\protect{\cite{samuel}}]
\label{samuelthm}
Let $k$ be an algebraically closed field, and let $R$ be a finitely generated $k$-algebra that is a domain. Then $R^*/k^*$ is a finitely generated free abelian group. 
\end{theorem}
\begin{proof}
Let $\overline{R}$ be the homogenization of R with respect to some new variable $x_0$. If $\overline{X} = \Proj(\overline{R})$ is normal, then the reasoning above shows that $R^*/k^*$ embeds in the finitely generated free abelian group $\Div^0_\partial(\overline{X})$, and subgroups of finitely generated free abelian groups are again finitely generated free abelian.

If $\overline{X}$ is not normal, let $\widetilde{\overline{X}}$ be the normalization of $\overline{X}$. The normalization map $\widetilde{\overline{X}} \xrightarrow{\eta} \overline{X}$ identifies $\eta^{-1}(X)$ with $\Spec(\widetilde{R})$, where $\widetilde{R}$ is the integral closure of $R$ in its fraction field. This gives an inclusion map $R^*/k^* \hookrightarrow (\widetilde{R})^*/k^*$. As $(\widetilde{R})^*/k^*$ is finitely generated free abelian by the previous case, $R^*/k^*$ is as well.
\end{proof}

\begin{remark}
Note that the unit group of the coordinate ring of a projective variety is trivial to compute: indeed, in this case $\overline{R}^* = k^*$, as any positively graded domain has units concentrated in degree $0$. Thus \Cref{samuelthm} is only interesting for rings which are not positively graded.
\end{remark}

\begin{remark}
The assumptions in \Cref{samuelthm} are necessary: if $k$ is not algebraically closed, then the unit group modulo scalar units may have torsion, i.e. roots of unity. If $R$ is not a domain, then $R^*/k^*$ need not be $\mathbb{Z}$-free: e.g. $R = k[x]/(x^2)$ has $R^*/k^*$ isomorphic to the additive group of $k$.
\end{remark}

\begin{remark}
In the setting of \Cref{samuelthm}, the exact sequence \eqref{eq:1} splits (since $R^*/k^*$ is free abelian), i.e. $R^* \cong k^* \oplus R^*/k^*$. Thus we also understand $R^*$ if we understand $R^*/k^*$.
\end{remark}

\subsection{Intrinsic tropicalization}
We now discuss some motivation for computing unit groups coming from tropical geometry, following the presentation in \cite{tropicalbook}. Recall that a variety $X$ is said to be \emph{very affine} if $X$ admits a closed embedding into an algebraic torus $\mathbb{T}$. Intuitively, a subvariety of $\mathbb P^m$ is affine if it misses a coordinate hyperplane, and very affine if it misses all coordinate hyperplanes. Algebraically, this means that the coordinate ring $R$ of $X$ is (isomorphic to) a quotient of a Laurent polynomial ring $k[x_1^\pm, \ldots, x_m^\pm]$. We note that given a very affine variety $X \subseteq \mathbb{T}^n$, one can take its projective closure $\overline{X} \subseteq \mathbb{P}^n$ with boundary $\partial X := \overline X \setminus X = \overline X \cap V(x_0\cdots x_n)$, and the above discussion (cf. \Cref{homogenizationLemma}, \Cref{div0Boundary}) carries over to this setting.

In general, there are many different closed embeddings of $X$ into tori $\mathbb{T}^m$ for various $m$. To remove the dependence on the choice of embedding, one must choose a ``natural" embedding of $X$ into a fixed torus. As it turns out, the right object to consider is the so-called \textit{intrinsic torus} of $X$, which is by definition \protect{\cite[Definition~6.4.2]{tropicalbook}}
\[
\mathbb{T}_{in} := \Hom_{\mathbb{Z}}(R^*/k^*, k^*).
\]
Note that by \Cref{samuelthm}, $R^*/k^*$ is free abelian, so the Hom group is isomorphic to a product of copies of $k^*$, which is an algebraic torus over $k$. A $\mathbb{Z}$-basis $f_1, \ldots, f_n$ of $R^*/k^*$ gives rise to an embedding $i : X \hookrightarrow T_{in}$, via $x \mapsto (f_1(x), \dots, f_n(x))$. With such a choice of basis, the importance of the intrinsic torus is immediate from the following ``pseudo-universal" property (cf. \protect{\cite[Proposition~6.4.4]{tropicalbook}}): for every closed embedding $j \colon X \xhookrightarrow{} \mathbb{T}^m$ of $X$ into a torus, there is a map of tori $\varphi: \mathbb{T}_{in} \rightarrow \mathbb{T}^m$ given by Laurent monomials (which need not be an embedding) such that the following diagram commutes: 

\[
\begin{tikzcd}
X \arrow[r, hook, "i"] \arrow[dr, hook, "j"]
& \mathbb{T}_{in} \arrow[d, "\varphi"]\\
& \mathbb{T}^m
\end{tikzcd}
\]

It is a basic task in tropical geometry to tropicalize a very affine variety with respect to a particular embedding in a torus. From a foundational viewpoint, it is desirable to have an \emph{intrinsic tropicalization}, with respect to the intrinsic torus, so that the tropicalization depends only on the very affine variety $X$ and not the specific embedding $X \hookrightarrow \mathbb{T}^m$. Furthermore, in the setup of the commutative diagram above, the tropicalization of $X$ embedded in $\mathbb T^m$ is given by the image of the intrinsic tropicalization under the affine map $\Trop(\varphi)$. Hence any other tropicalization of $X$ can be recovered from the intrinsic tropicalization.

However, from a computational standpoint, the very affine variety is most often described by its ideal in a fixed embedding. To obtain an intrinsic tropicalization one must be able to compute the defining ideal of the very affine variety in its intrinsic torus; the key to doing so is to first compute a basis of $R^*/k^*$. 
Of course an embedding $i$ into the intrinsic torus depends on our choice of basis for $R^*/k^*$, but we nevertheless often speak of \emph{the} intrinsic embedding into the intrinsic torus. 

\section{General results on varieties}
\label{preliminaryresults}

In this section we reinterpret our problem in the context of class groups. We retain the setup from the previous section: let $X$ be a very affine variety over an algebraically closed field $k$, with coordinate ring $R$.

\begin{defn}
Define $\Cl_{\partial}^0(\overline{X})$ to be the cokernel of the group homomorphism $R^*/k^* \xrightarrow{\quad \phi \quad} \Div_{\partial}^0(\overline{X})$. 
\end{defn}

By definition, there is a short exact sequence of abelian groups

\begin{equation} \label{mainexactsequence}
1 \xrightarrow{\qquad} R^*/k^* \xrightarrow{\quad \phi \quad} \Div_{\partial}^0(\overline{X}) \xrightarrow{\qquad} \Cl_{\partial}^0(\overline{X}) \xrightarrow{\qquad} 0
\end{equation}

\begin{coro} \label{firstUpperBound}
Let $r$ be the number of divisorial components of $\partial X$. Then $\rank R^*/k^* \leq r - 1$, with equality if and only if $\Cl^0_\partial(\overline{X})$ is torsion.
\end{coro}
\begin{proof}
The subgroup $\Div_\partial(\overline{X})$ of $\Div(\overline{X})$ (consisting of Weil divisors supported on $\partial X$) is a free group of rank $r$, and the degree $0$ condition implies $\Div^0_\partial(\overline{X})$ is a free subgroup of rank $r-1$.
\end{proof}

\begin{coro} \label{curveboundcorollary}
If $C$ is a very affine curve over $k$ with coordinate ring $R$, with projective closure $\overline{C} \subseteq \mathbb{P}_k^n$ of degree $d$, then $\rank R^*/k^* \leq (n+1)d - 1$. 
\end{coro}

\begin{proof}
As $C$ is a curve, the divisorial components of $\partial C$ are just the (closed) points of $\partial C$. Since $C$ is very affine, the boundary $\partial C$ consists of the intersections of $\overline{C}$ with each of the $n+1$ coordinate hyperplanes in $\mathbb{P}^n_k$. Then $\deg \overline{C} = d$ implies $\partial C$ consists of at most $(n+1)d$ points, and the result follows from \Cref{firstUpperBound}.
\end{proof}

 Samuel's \Cref{samuelthm} tells us that the structure of the unit group -- as an \textit{abstract} group -- is as nice as possible. However, we need more information about the other groups in \eqref{mainexactsequence} to explicitly give generators for $R^*/k^*$. The following basic, but crucial, point states that all relations in $\Cl_{\partial}^0(\overline{X})$ are ``geometric", in the sense that they come from the class group of $\overline{X}$.
\begin{prop}
\label{clsubgroup}
$\Cl_{\partial}^0(\overline{X})$ is a subgroup of $\Cl^0(\overline{X})$.
\end{prop}
\begin{proof}
Consider the composition
\[
\Cl_{\partial}^0(\overline{X}) \cong 
\Div_{\partial}^0(\overline{X})/(R^*/k^*) \xhookrightarrow{\alpha} \Div^0(\overline{X})/(R^*/k^*) \xrightarrow{\beta} \Cl^0(\overline{X}).
\]
To show that the composite is an injection, it suffices to show that $\img(\alpha) \cap \ker(\beta) = \{0\}$. But this follows since $\ker(\beta) = \Frac(\overline{R})^*/(R^*/k^*)$, and $\Frac(\overline{R})^* \cap \Div_{\partial}^0(\overline{X}) = R^*/k^*$, as a rational function on $\overline{X}$ supported only on $\partial X$ is a unit on $X$.
\end{proof}

\begin{remark}
Recall that the class group of the ring of integers of a number field is finite. If a similar result held in our setting, \Cref{firstUpperBound} would give an explicit description for the rank of $R^*/k^*$. Unfortunately, of course, $\Cl(\overline X)$ need not be so well-behaved in general.
\end{remark}

In general, our approach to computing $R^*/k^*$ via \eqref{mainexactsequence} proceeds in three parts:

\begin{restatable}{question}{questionone}
\label{questionone}
What are the generators of the image of $R^*/k^*$ in $\Div_{\partial}^0 (\overline{X})$?
\end{restatable}

\begin{restatable}{question}{questiontwo}
\label{questiontwo}
Given $D \in \Div_{\partial}^0 (\overline{X})$ that is in the image of $R^*/k^*$, can we find polynomials $f, g$ such that $f/g \in R^*/k^*$ is mapped to $D$ (under the inclusion $R^* \subseteq \Frac(\overline{R})$)?
\end{restatable}

\begin{restatable}{question}{questionthree}
\label{questionthree}
Given an element of $R^*/k^*$ expressed as a rational function as in \Cref{questiontwo}, can we find a representative for it in $R$?
\end{restatable}

Note that \Cref{clsubgroup} suggests a path towards progress on \Cref{questionone}, as the image of $R^*/k^*$ in $\Div^0_\partial(\overline{X})$ equals $\ker(\Div^0_\partial(\overline{X}) \to \Cl^0_\partial(\overline{X}))$, and by \Cref{clsubgroup} this is the same as $\ker(\Div^0_\partial(\overline{X}) \to \Cl^0(\overline{X}))$. Ultimately though, one needs control over $\Cl^0(\overline{X})$ to solve Questions \ref{questionone} and \ref{questiontwo}, and this will require methods particular to the varieties under consideration.

On the other hand, \Cref{questionthree} can be solved with relatively basic Gr\"obner basis algorithms, which we use repeatedly in the remainder of the paper. We note that ordinary Gr\"obner basis arguments over polynomial rings can be adapted to Laurent polynomial rings by identifying the rings $k[x_1^\pm, \ldots, x_n^\pm] \cong k[x_1, \ldots, x_n, t]/(tx_1 \ldots x_n - 1)$.

\begin{algorithm}[Clearing denominators]
\label{clearDenom}
\begin{algorithmic}[1]
\item[]
\REQUIRE $f, g \in k[x_1^{\pm 1}, \ldots, x_n^{\pm 1}]$, $I = (\phi_1, \ldots, \phi_m) \subseteq k[x_1^{\pm 1}, \ldots, x_n^{\pm 1}]$ with a fixed monomial order
\ENSURE $h \in k[x_1^{\pm 1}, \ldots, x_n^{\pm 1}]$ with $f - gh \in I$ if such an $h$ exists, or \FALSE \, otherwise
\STATE $J \leftarrow I + (g)$
\STATE $G \leftarrow$ Gr\"obnerBasis($J$)
\IF{$f \notin$ ideal($G$) }
    \RETURN \!\!\! \FALSE
\ENDIF
\STATE $C = (C_0,\ldots, C_m) \leftarrow$ a vector with entries in $R$ such that $f = C_0 g + C_1 \phi_1 + \ldots + C_m \phi_m$
\RETURN{\!$C_0$}
\end{algorithmic}
\end{algorithm}

\begin{lemma}
For $f, g \in R=k[x_1^{\pm 1}, \ldots, x_n^{\pm 1}]/I$, \Cref{clearDenom} correctly determines whether there exists $h \in R$ such that $f = gh$, and returns such an $h$ if it exists.
\end{lemma}
\begin{proof}
A standard Gr\"obner basis argument checks whether $f\in J$ and, if so, finds such a vector $C$ as above. Note that $f \in J$ if and only if there exists $h$ such that $f - gh \in I$, so that $f = gh\in R$.
\end{proof}

\begin{algorithm}[Testing units]
\label{testUnits}
\begin{algorithmic}[1]
\item[]
\REQUIRE $h \in k[x_1^{\pm 1}, \ldots, x_n^{\pm 1}]$, $I = (\phi_1, \ldots, \phi_m) \subseteq k[x_1^{\pm 1}, \ldots, x_n^{\pm 1}]$ with a fixed monomial order
\ENSURE \TRUE \, if $h \in (k[x_1^{\pm 1},\dots,x_n^{\pm 1}]/I)^*$, or \FALSE \, otherwise
\STATE $J\leftarrow I + (h)$
\STATE $G\leftarrow$ Gr\"obnerBasis($J$)
\IF{$1 \in$ ideal($G$) }
    \RETURN \!\!\! \TRUE
\ENDIF
\RETURN \!\!\! \FALSE
\end{algorithmic}
\end{algorithm}

\begin{lemma}
For $h \in R=k[x_1^{\pm 1}, \ldots, x_n^{\pm 1}]/I$, \Cref{testUnits} correctly tests if $h$ is a unit in $R$.
\end{lemma}
\begin{proof}
A standard Gr\"obner basis argument checks whether $1\in J$. Note that $1 \in J = I + (h) \subseteq k[x_1^{\pm 1},\dots,x_n^{\pm 1}]$ if and only if $h \in (k[x_1^{\pm 1},\dots,x_n^{\pm 1}]/I)^*$.
\end{proof}

\begin{algorithm}[Computing preimages of $R^* \rightarrow \Frac(\overline{R})^*$]
\label{computingRepOfUnit}
\begin{algorithmic}[1]
\item[]
\REQUIRE $\overline{f}, \overline{g} \in k[x_0,\dots,x_n]$ homogeneous, $\frac{\overline{f}}{\overline{g}} \in \Frac(\overline{R})^*$ and $I = (\phi_1,\dots,\phi_m)\subseteq k[x_1^{\pm 1}, \ldots, x_n^{\pm 1}]$ with a fixed monomial order 
\ENSURE $h \in k[x_1^{\pm 1},\dots,x_n^{\pm 1}]$ such that $h = \frac{\overline{f}}{\overline{g}}$ in $\Frac(\overline{R})^*$ (via the inclusion $R^* \subseteq \Frac(\overline{R})^*$) if such an $h$ exists, or \FALSE \, otherwise
\STATE{$f \leftarrow \overline{f}(1,x_1,\dots,x_n)$}
\STATE{$g \leftarrow \overline{g}(1,x_1,\dots,x_n)$}
\IF{\Cref{clearDenom}$(f,g,I) = $ \FALSE}
    \RETURN \FALSE
\ELSE
    \STATE{$h \leftarrow \Cref{clearDenom}(f,g,I)$}
    \IF{\Cref{testUnits}$(h, I) =$ \TRUE}
    	\RETURN $h$
    \ELSE
    	\RETURN \FALSE
    \ENDIF
\ENDIF
\end{algorithmic}
\end{algorithm}

\begin{lemma}
\label{questionthreelemma}
Let $X$ be a very affine variety over $k$ with coordinate ring $R = k[x_1^{\pm 1}, \ldots, x_n^{\pm 1}]/(\phi_1, \ldots, \phi_m)$. Let $\overline f$ and $\overline g$ be homogeneous polynomials in $k[x_0,\ldots, x_n]$, and $f, g \in k[x_1,\dots,x_n]$ their dehomogenizations with respect to $x_0$. Given a rational function $\frac{\overline{f}}{\overline{g}} \in \Frac(\overline{R})^*$, \Cref{computingRepOfUnit} correctly decides whether $\frac{f}{g} \in R^*$ (via the inclusion $R^* \subseteq \Frac(\overline{R})^*$), and if so, computes a representative $h \in k[x_1^{\pm 1},\dots,x_n^{\pm}]$ for $\frac{f}{g}$.
\end{lemma}
\begin{proof}
If $\overline{f}/\overline{g} \in R^*$ then there must exist a Laurent polynomial $h \in R^*$ such that $\frac{\overline{f}}{\overline{g}} = \overline{h}$ in $\Frac(\overline{R})^*$, where $\overline h$ is the homogenization of $h$ with respect to $x_0$. Thus $\overline{f} - \overline{g}\overline{h} = 0$ in $\Frac(\overline{R})^*$, so $f - gh \in I$. Since $h \in R^*$, \Cref{testUnits} will verify that $h$ is a unit, and \Cref{computingRepOfUnit} will return $h$. 

Now assume that $\overline{f}/\overline{g} \notin R^*$. The algorithm will return false unless \Cref{clearDenom} returns some $h \in R^*$ such that $f - gh$. Suppose this occurs. By homogenizing, we see that $\overline{f} = \overline{g}\overline{h}$ in $\overline{R}$ and $\frac{\overline{f}}{\overline{g}} = \overline{h}$ in $\Frac(\overline{R})^*$, which is a contradiction.
\end{proof}

\section{Fermat curves and plane conics} \label{planequadrics}

We now consider two simple families of curves, Fermat curves and plane conics. These serve as our first two classes of examples for our general problem of computing unit groups.

\subsection{Fermat curves}

We first approach the problem of constructing unit groups in a purely elementary, algebraic way:
\begin{lemma}
\label{fermatlemma}
Let $T \coloneqq k[x_1^{\pm 1},\dots,x_d^{\pm 1}]$ be a Laurent polynomial ring, $I \subseteq T$ an ideal, $u \in T$ a monomial, $a \in k^*$, and $f \in I$. If there exist $g, h \in T$ with $f + au = gh$, then $\overline{g}, \overline{h}$ are units in $R \coloneqq T/I$.
\end{lemma}

\begin{proof}
Note that $u$ is a unit in $T$ (being monomial), so $a \overline{u}$ is a unit in $R$. Since $\overline{g}\overline{h} = a \overline{u} \in R^*$, we have that $\overline{g}$ and $\overline{h}$ are also units in $R$. 
\end{proof}

\begin{exam}[Fermat curves]
\label{fermatcurvesoln}
Consider the family of Fermat curves, which are plane curves in $\mathbb{P}^2 = \Proj(k[x,y,z])$ defined by equations of the form $x^d + y^d = z^d$, for $d \in \mathbb{N}$. For a fixed degree $d$, we have $\overline{C} \coloneqq V(x^d + y^d - z^d) \subseteq \mathbb{P}^2$ with homogeneous coordinate ring $\overline{R} \coloneqq \mathbb{C}[x,y,z]/(x^d + y^d - z^d)$. Dehomogenizing with respect to $z$ and intersecting with the torus in $\mathbb{A}^2$ gives a very affine Fermat curve $C$ with coordinate ring $R = \mathbb{C}[x^{\pm 1}, y^{\pm 1}]/(x^d + y^d - 1)$. 

We will use \eqref{mainexactsequence} and \Cref{fermatlemma} to show that the unit group $R^*/k^*$ has $3d-1$ independent elements. By \Cref{curveboundcorollary}, $\rank R^*/k^* \leq (n+1)d - 1 = 3d-1,$
so this bound is tight.

Consider the relation
\[
    -x^d = y^d - 1 = \prod_{i = 0}^{d-1}(y - \zeta_d^i)
\]
which holds in $R$, where $\zeta_d$ is a primitive $d$-th root of unity. From \Cref{fermatlemma}, we conclude that $(y - \zeta_d^i)$ is a unit in $R$, for all $0 \le i \le d-1$. 
Interpreting the above relation as a dependency among $x, y - \zeta_d, \ldots, y - \zeta_d^{d-1}$ in $R^*/k^*$, we can write any $y - \zeta_d^i$ multiplicatively in terms of $x$ and $y - \zeta_d^j$ for $j \ne i$. Thus we can choose -- for instance -- to treat $y - \zeta_d^{d-1}$ as redundant, and we obtain new units $y - \zeta_d^i$ for $0 \le i \le d - 2$. Note that the relation above does not give a way to express $x$ in terms of $y - \zeta_d^i$, since $x$ appears with multiplicity $d$.

In an analogous way, we may also rearrange the defining equation of $R$ to obtain
\[
    -y^d = x^d - 1 = \prod_{i = 0}^{d-1}(x - \zeta_d^i)
\]
which gives new units $x - \zeta_d^i$ for $0 \leq i \leq d-2$. Finally, the rearrangement
\[
    1 = x^d + y^d = \prod_{i = 0}^{d-1}(x - \zeta_{2d}^{2i+1}y)
\]
gives new units $x - \zeta_{2d}^{2i+1}y$ for $0 \leq i \leq d-2$.

We thus have the units $x - \zeta_d^i$, $y - \zeta_d^i$, $x - \zeta_{2d}^{2i+1}$ where $0 \le i \le d-2$. In addition to the two units $x, y$, this gives a total of $3(d-1) + 2 = 3d - 1$ units. Note that although we have accounted for obvious redundancies by removing $x - \zeta_d^{d-1}, y - \zeta_d^{d-1}$, and $x + \zeta_{2d}^{2d-1}y$, we have not yet shown that these $3d-1$ units are independent. Algebraically, this would entail showing that there are no nontrivial multiplicative relations between these $3d-1$ elements, a fairly nontrivial task. We instead adopt a geometric approach, whose utility will become evident already in this case.

First, the divisors of these units (viewed as rational functions) are supported on the boundary $\partial C$ of the Fermat curve, which consists of the following $3d$ points:
\begin{enumerate}
\item $P_i \coloneqq [\zeta_{2d}^{2i+1} \colon 1 \colon 0]$ for $0 \leq i \leq d-1$
\item $Q_i \coloneqq [\zeta_d^i \colon 0 \colon 1]$ for $0 \leq i \leq d-1$
\item $T_i \coloneqq [0 \colon \zeta_d^i \colon 1]$ for $0 \leq i \leq d-1$
\end{enumerate}
As before, let $\phi : R^*/k^* \rightarrow \Div_{\partial}^0(\overline{X})$ be the injection in \Cref{mainexactsequence}. We have
\begin{enumerate}
\item $\phi(x) = \sum T_i - \sum P_i$
\item $\phi(y) = \sum Q_i - \sum P_i$
\item $\phi(y - \zeta_d^j) = dT_j - \sum P_i$ for $0 \leq j \leq d-2$
\item $\phi(x - \zeta_d^j) = dQ_j - \sum P_i$ for $0 \leq j \leq d-2$
\item $\phi(x - \zeta_{2d}^{2j+1}y) = (d-1)P_j - \sum_{i\neq j} P_i$ for $0 \leq j \leq d-2$
\end{enumerate}

Under the identification $\Div_{\partial}(\overline{C}) = \mathbb{Z}\langle P_1,\dots, P_d,Q_1,\dots,Q_d,T_1,\dots,T_d \rangle \cong \mathbb{Z}^{3d}$, we obtain the following $3d \times (3d-1)$ matrix whose columns represent the divisors of our given units.
\begin{figure}[h!]
\begin{center}
\begin{tikzpicture}
\draw (0,0) rectangle node{$1_{d\times 1}$} (1,5);
\draw (1,0) rectangle node{$0_{d\times 1}$} (2,5);
\draw (0,5) rectangle node{$0_{d\times 1}$} (1,10);
\draw (1,5) rectangle node{$1_{d\times 1}$} (2,10);
\draw (2,0) rectangle node{$0_{1\times d-1}$} (6,1);
\draw (2,1) rectangle node{$dI_{d-1}$} (6,5);
\draw (2,5) rectangle node{$0_{d\times d-1}$} (6,10);
\draw (6,0) rectangle node{$0_{d+1\times d-1}$} (10, 6);
\draw (6,6) rectangle node{$d I_{d-1}$} (10, 10);
\draw (10,0) rectangle node{$0_{2d\times d-1}$} (15, 10);
\draw (0,10) rectangle node{$-1_{1\times 3d-1}$} (15,11);
\draw (0,11) rectangle node{$-1_{d-1\times 2d}$} (11, 15);
\draw (11,11) rectangle node{$-1_{d-1\times d-1}+dI_{d-1}$} (15, 15);
\end{tikzpicture}
\end{center}
\caption{The block matrix whose columns are divisors of the units described in Example \ref{fermatcurvesoln} for the Fermat curve $x^d + y^d = z^d$. Here $a_{m \times n}$ is an $m \times n$ matrix whose elements are all $a$, and $I_n$ is the $n \times n$ identity matrix.}
\label{fermat-divisor-matrix}
\end{figure}
% (for $a \in k$ and $m, n \in \mathbb{N}$)
\begin{comment}
\begin{figure}[h!]
\label{fermat-divisor-matrix}
\begin{center}
\begin{tikzpicture}
\draw (0,0) rectangle node{$-1_{d-1
\times d-1} + dI_{d-1}$} (3,3);
\draw (0,3) rectangle node{$-1_{2d \times d-1}$} (3,9);
\draw (3,0) rectangle node{$-1_{3d-1 \times 1}$} (5,9);
\draw (5,0) rectangle node{$0_{d-1 \times 2d}$} (12,3);
\draw (5,3) rectangle node{$dI_{d-1}$} (7,5);
\draw (7,3) rectangle node{$0_{d-1 \times d+1}$} (12,5);
\draw (5,5) rectangle node{$0_{d-1 \times d}$} (8,7);
\draw (8,5) rectangle node{$dI_{d-1}$} (10,7);
\draw (10,5) rectangle node{$0_{d-1 \times 1}$} (12,7);

\draw (5, 7) rectangle node{$1_{1 \times d}$} (8,8);
\draw (8,7) rectangle node{$0_{1 \times d}$} (12,8);
\draw (5,8) rectangle node{$0_{1 \times d}$} (8,9);
\draw (8,8) rectangle node{$1_{1 \times d}$} (12,9);

\end{tikzpicture}
\end{center}
\caption{The block matrix whose rows are divisors of the units described in Example \ref{fermatcurvesoln} for the Fermat curve $x^d + y^d = z^d$. Here $a_{m \times n}$ is an $m \times n$ matrix whose elements are all $a$ (for $a \in k$ and $m, n \in \mathbb{N}$), and $I_n$ is the $n \times n$ identity matrix.}
\end{figure}
\end{comment}

A straightforward check shows that this matrix has full rank $3d-1$, and therefore our units have no relations.
It is natural at this point to ask whether these units form a basis for the unit group. It turns out that this need not be the case, as shown in \Cref{fermat2example}.
\end{exam}

\begin{remark}
We observe several things about this computation. First, we did not necessarily compute generators of $R^*/k^*$. Instead, we found enough mutually independent elements to confirm a rank statement on $R^*/k^*$. Next, this technique was only effective for the Fermat curve because of special features of its defining equation. With more variables or nearly any perturbation of the defining equation, the method of obtaining units above fails. Finally, the argument above can only prove lower bounds on the rank of the unit group. We want to compute generators of the unit group, so in general we will need more tools than \Cref{fermatlemma}. 
\end{remark}

\subsection{Plane conics}

Let $\overline{C} \subseteq \mathbb P^2_k$ be a smooth projective plane conic defined by a homogeneous quadric $f(x,y,z)$, and $C$ the corresponding very affine curve (obtained by dehomogenizing with respect to $z$ and intersecting with the 2-torus $\mathbb{T}^2 := \mathbb{A}^2 \setminus V(xy)$), with coordinate ring $R$. We describe methods for answering \Cref{questionone} and \Cref{questiontwo} in this case. Combined with \Cref{questionthreelemma}, this gives an algorithm to compute a basis of $C^*/k^*$.

\begin{algorithm}[Computing unit groups of conics]
\label{conicalg}
\item[]
\begin{algorithmic}[1]

\REQUIRE A homogeneous quadric $f(x,y,z)$ defining a plane conic $\overline C\subseteq \mathbb P^2$
\ENSURE A basis of $R^*/k^*$
\STATE $P_1, \ldots, P_n \leftarrow$ boundary points of $\overline C$
\STATE $P \leftarrow$ any other point of $\overline C$
\FORALL{$i \in \{1,\dots, n\}$}
	\STATE $L_i \leftarrow$ defining equation of line between $P_i$ and $P$
\ENDFOR
\FORALL{$i\in \{1,\dots, n-1\}$}
	\STATE Compute $f_i \in k[x^{\pm 1}, y^{\pm 1}]$ equivalent to $L_i/L_{i+1}$ in $R$ using \Cref{computingRepOfUnit}
\ENDFOR
\RETURN{$f_1,\ldots, f_{n-1}$}
\end{algorithmic}
\end{algorithm}

\begin{theorem}
Algorithm \ref{conicalg} computes a basis for $R^*/k^*$.
\end{theorem}
\begin{proof}
Observe that $\Cl^0(\overline C) = 0$ (as $\overline C \cong \mathbb P^1$). \eqref{mainexactsequence} then implies that the injection $R^*/k^*\hookrightarrow \Div_{\partial}^0(\overline C)$ is an isomorphism. Then, note that $P_1-P_2, \cdots, P_{n-1}-P_n$ forms a basis for $\Div_{\partial}^0(\overline C)$, and $L_i/L_{i+1}$ corresponds to the divisor $P_i-P_{i+1}$. Applying \Cref{computingRepOfUnit} finishes the proof.
\end{proof}

Note that the choice of basis $\{P_i - P_{i+1}\}$ in the above proof was arbitrary; any basis of $\Div_{\partial}^0(\overline C)$ would suffice. On the other hand, this basis gives the very simple rational functions $L_i/L_{i+1}$.

\begin{exam}
\label{fermat2example}
Consider the degree 2 Fermat curve $\overline{C}$ defined by $x^2 + y^2 = z^2$. We show that the units produced in \Cref{fermatcurvesoln} are not generators of $R^*/k^*$. As in \Cref{fermatcurvesoln}, we have the following boundary points:

\begin{enumerate}
\item $P_0 \coloneqq [i \colon 1 \colon 0]$
\item $P_1 \coloneqq [-i \colon 1 \colon 0]$
\item $Q_0 \coloneqq [1 \colon 0 \colon 1]$
\item $Q_1 \coloneqq [-1 \colon 0 \colon 1]$
\item $T_0 \coloneqq [0 \colon 1 \colon 1]$
\item $T_1 \coloneqq [0 \colon -1 \colon 1]$
\end{enumerate}

\Cref{fermatcurvesoln} gives the following units and divisors (with $R^*/k^* \xhookrightarrow{\phi} \Div_{\partial}^0(\overline{C})$ as in \Cref{mainexactsequence}):
\begin{enumerate}
\item $\phi(x) = T_0 + T_1 - P_0 - P_1$
\item $\phi(y) = Q_0 + Q_1 - P_0 - P_1$
\item $\phi(y - 1) = 2T_0 - P_0 - P_1$
\item $\phi(x - 1) = 2Q_0 - P_0 - P_1$
\item $\phi(x - iy) = P_0 - P_1$ 
\end{enumerate}

The subgroup of $\Div_{\partial}^0(\overline{C})$ generated by these divisors is given by the integer column span of the matrix, which is exactly \Cref{fermat-divisor-matrix} for $d = 2$:

\[
\begin{bmatrix}
    -1  & -1  & -1 & -1 & 1 \\
    -1  & -1  & -1 & -1 & -1 \\
    0   &  1  & 0  & 2  & 0 \\
    0   & 1   & 0  & 0  & 0 \\
    1   & 0   & 2  & 0  & 0 \\
    1   &  0  & 0  & 0  & 0
\end{bmatrix}
\]

As noted in Algorithm \ref{conicalg}, one basis for $\Div^0_\partial(\overline C)$ is $\{P_i - P_{i+1} \mid 1 \le i \le n-1\} = P_1-P_2, P_2-P_3, \ldots, P_{n-1}-P_n$. From this basis we obtain the matrix
\[
\begin{bmatrix}
    1  & 0 & 0 & 0 & 0 \\
    -1 & 1 & 0 & 0 & 0 \\
    0  & -1 & 1 & 0 & 0 \\
    0  & 0  & -1 & 1 & 0 \\
    0  & 0 & 0 & -1 & 1 \\
    0 & 0 & 0 & 0 & -1
\end{bmatrix}
\]
The first lattice has index 4 in the second. It follows that the units given in \Cref{fermatcurvesoln} are not generators in this case.
\end{exam}

\begin{exam}
\label{puiseux}
Let $\overline C$ be the conic defined by $f = (1+t)x^2 + (1+t)y^2 + (1+t)z^2 - (2+2t+t^2)xy - (2+2t+t^2)yz - (2+2t+t^2)xz$, where $k$ is the field of Puiseux series in $t$ over $\mathbb{C}$. Consider the very affine curve $C$ given by intersecting with the canonical torus. Its boundary points are
\begin{enumerate}
  \item $P_1 \coloneqq [0:1:t+1]$
  \item $P_2 \coloneqq [0:t+1:1]$
  \item $P_3 \coloneqq [1:0:t+1]$
  \item $P_4 \coloneqq [t+1:0:1]$
  \item $P_5 \coloneqq [1:t+1:0]$
  \item $P_6 \coloneqq [t+1:1:0]$
\end{enumerate}

As described above, we can take a basis of $\Div_{\partial}^0(\overline{C})$ to be differences of these boundary points, e.g. $P_3 - P_1$, $P_3 - P_2$, $P_5-P_3$, $P_5-P_4$, and $P_6-P_1$. \Cref{conicalg} gives the following particularly nice generators of the unit group:

\begin{enumerate}
\item $P_3 - P_1$ gives $f_1 \coloneqq$ (line between $P_3$ and $P_2$)/(line between $P_1$ and $P_2$) $= \frac{(t+1)^2x + y - (t+1)}{x} = (t+1)^2 + yx^{-1} - (t+1)x^{-1}$
\item $P_3 - P_2$ gives $f_2 \coloneqq$ (line between $P_1$ and $P_3$)/(line between $P_1$ and $P_2$) $= \frac{(t+1)x + (t+1)y - 1}{x} = (t+1) + (t+1)yx^{-1} - x^{-1}$
\item $P_5-P_3$ gives $f_3 \coloneqq$ (line between $P_5$ and $P_4$)/(line between $P_3$ and $P_4$) $= \frac{(t+1)x - y - (t+1)^2}{y} = (t+1)xy^{-1} - 1 - (t+1)^2y^{-1}$
\item $P_5 - P_4$ gives $f_4 \coloneqq$ (line between $P_5$ and $P_3$)/(line between $P_3$ and $P_4$) $= \frac{(t+1)x - y - 1}{y} = (t+1)xy^{-1} - 1 - y^{-1}$
\item $P_6 - P_1$ gives $f_5 \coloneqq$ (line between $P_6$ and $P_2$)/(line between $P_1$ and $P_2$) $= \frac{x - (t+1)y + (t+1)^2}{x} = 1 - (t+1)yx^{-1} + (t+1)^2x^{-1}$
\end{enumerate}

So the intrinsic torus has dimension $5$, and these generators specify a map into the intrinsic torus, corresponding to the ring map $\varphi \colon k[x_1^{\pm 1},\dots,x_5^{\pm 1}] \rightarrow k[x^{\pm 1}, y^{\pm 1}]/(f)$ sending $x_i \mapsto f_i$. 

We note that the tropicalization of $f$ is simply the tropical line $0\oplus x \oplus y$ shown in Figure \ref{bad-tropicalization-of-conic}:
\begin{figure}[h!]
\begin{center}
\begin{tikzpicture}

\draw (2,2) -- (0.5,0.5);
\draw (4,2) -- (2,2);
\draw (2,4) -- (2,2);
\end{tikzpicture}
\end{center}
\caption{The tropicalization of the conic in \Cref{puiseux}.}
\label{bad-tropicalization-of-conic}
\end{figure}
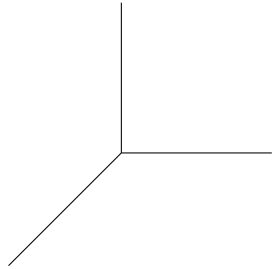

We used Singular \cite{DGPS} to compute the tropicalization of $f$ in its intrinsic torus with basis equal to $\{x, y, f_1, f_2, f_3\}$. %explain why this is a basis?
The intrinsic tropicalization has the following snowflake structure typical of a generic tropical conic as in Figure \ref{good-tropicalization-of-conic}:

% We used Macaulay2 to compute the kernel $I$ of this map, noting that the image of our quadric in the intrinsic torus will be $\Spec(k[x_1^{\pm 1},\dots,x_5^{\pm 1}]/I)$. We then use Gfan to compute the tropicalization of $I$, which has the following snowflake structure.

\begin{figure}[h!]
\begin{center}
\begin{tikzpicture}

\draw (3,4) -- (3,5);
\draw (3,4) -- (4,4);
\draw (2,3) -- (3,4);
\draw (1,3) -- (2,3);
\draw (2,2) -- (2,3);
\draw (0,2) -- (1,3);
\draw (1,1) -- (2,2);
\draw (1,3) -- (1,5);
\draw (2,2) -- (4,2);
\end{tikzpicture}
\end{center}
\caption{The intrinsic tropicalization of the conic in \Cref{puiseux}.}
\label{good-tropicalization-of-conic}
\end{figure}
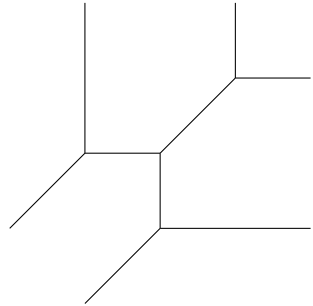
\end{exam}

\begin{remark}
Consider the complete graph whose nodes are the elements of $\partial C$. Choose a spanning tree of this graph, and pick an edge for each direction. Each edge of this tree gives a divisor; namely an edge from $P$ to $Q$ gives the divisor $P - Q$. This gives a basis of $\Div_{\partial}(\overline{C})$.

For instance, in \Cref{fermat2example}, the basis
\[
P_0 - Q_0, P_1 - Q_0, Q_1 - P_0, T_0 - P_0, T_1 - P_0
\]
corresponds to the directed tree in Figure \ref{tree1-for-conic-example}:

\begin{figure}[h!]
\begin{center}
\begin{tikzpicture}[node distance = {1.0cm and 1.5cm}, v/.style = {draw, circle},> = stealth, % arrow head style
            shorten > = 1pt, % don't touch arrow head to node
            auto,
            semithick % line style
        ]
  % \graph [spring layout, nodes = {}, horizontal = A to C]
  % {
  %   A -- C -- E,
  %   B -- {D -- F}
  % };

  \node (a) [v] {$P_0$};
  \node (c) [v, right = of a] {$Q_0$};
  \node (e) [v, right = of c] {$P_1$};
  
  \node (b) [v, below = of a] {$T_0$};
  \node (d) [v, right = of b] {$Q_1$};
  \node (f) [v, right = of d] {$T_1$};
  
  \path[->] (a) edge node {} (c);
  \path[->] (e) edge node {} (c);
  \path[->] (d) edge node {} (a);
  \path[->] (b) edge node {} (a);
  \path[->] (f) edge node {} (a);

\end{tikzpicture}
\end{center}
\caption{A directed tree describing a basis for the intrinsic torus of Example \ref{fermat2example}.}
\label{tree1-for-conic-example}
\end{figure}
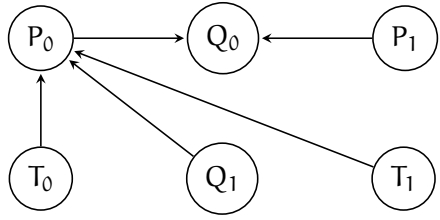

Similarly, the basis
\[
Q_0 - P_0, Q_0 - P_1, Q_0 - Q_1, Q_0 - T_0, Q_0 - T_1
\]

corresponds to the tree in Figure \ref{tree2-for-conic-example} (rooted at $Q_0$):

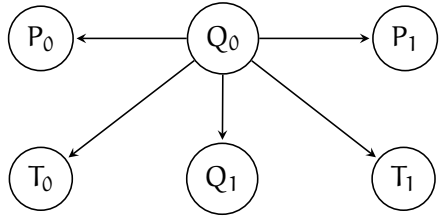
\begin{figure}[h!]
\begin{center}
\begin{tikzpicture}[node distance = {1.0cm and 1.5cm}, v/.style = {draw, circle},> = stealth, % arrow head style
            shorten > = 1pt, % don't touch arrow head to node
            auto,
            semithick % line style
        ]
  % \graph [spring layout, nodes = {}, horizontal = A to C]
  % {
  %   A -- C -- E,
  %   B -- {D -- F}
  % };

  \node (a) [v] {$P_0$};
  \node (c) [v, right = of a] {$Q_0$};
  \node (e) [v, right = of c] {$P_1$};
  
  \node (b) [v, below = of a] {$T_0$};
  \node (d) [v, right = of b] {$Q_1$};
  \node (f) [v, right = of d] {$T_1$};
  
  \path[->] (c) edge node {} (a);
  \path[->] (c) edge node {} (b);
  \path[->] (c) edge node {} (d);
  \path[->] (c) edge node {} (e);
  \path[->] (c) edge node {} (f);

\end{tikzpicture}
\end{center}
\caption{Another directed tree describing a basis for the intrinsic torus of Example \ref{fermat2example}.}
\label{tree2-for-conic-example}
\end{figure}

\end{remark}

\section{Rational Normal Curves}
\label{rationalnormalcurves}
We next turn our attention to rational normal curves in parametric form. Recall that for any $n$, a rational normal curve $\overline{C}$ of degree $n$ is the image of $\mathbb{P}^1$ under an embedding $\nu \colon \mathbb{P}^1 \xhookrightarrow{} \mathbb{P}^n$ given by $\nu([S \colon T]) = [f_0(S,T):\dots:f_n(S,T)]$, where $f_0, \dots,f_n$ are $k$-linearly independent homogeneous polynomials of degree $n$. 
Let $C := \overline{C} \cap \mathbb{T}^n$ be the corresponding very affine curve, with coordinate ring $R$. Our goal in this section is to give an algorithm for computing a basis of $R^*/k^*$.

\begin{remark}
Plane conics are precisely the rational normal curves of degree 2, so the following discussion generalizes part of Section \ref{planequadrics} in some sense. Note though that the presentation of the curves in question has changed: here we do not begin with the implicit equations of the rational normal curve in $\mathbb{P}^n$.
\end{remark}

The following is a modification of the polynomial subalgebra membership algorithm given in \protect{\cite[7.3.7]{coxsheaolittle}}.
\begin{algorithm}[Subalgebra membership]
\label{subalgmembership}
\item[]
\begin{algorithmic}[1]
\REQUIRE{$f_0,\dots,f_n$ degree $n$ homogeneous polynomials in $k[S,T]$ defining a rational normal curve, and a rational function $\frac{f}{g} \in k(S,T)$}
\ENSURE{$\gamma \in k[x_1^{\pm 1},\dots,x_n^{\pm 1}]$ such that its homogenization $\overline \gamma\in k[x_0^{\pm 1},\dots, x_n^{\pm 1}]$ satisfies $\frac{f(S,T)}{g(S,T)} = \overline{\gamma} \big (f_0(S,T),\dots,f_n(S,T)\big )$ if such a $\gamma$ exists, or \FALSE \, otherwise}
\STATE{$I \gets ideal(y_0 - f_0,\dots,y_n - f_n, z_0 - s_0,\dots,z_n - s_n, f_0 s_0 - 1,\dots, f_n s_n - 1, gs - 1)$ in the polynomial ring $k[y_0,\dots,y_n,z_0,\dots,z_n,s_0,\dots,s_n,s,S,T]$}
\STATE{$G \gets$ Gr\"obnerBasis(I) in a monomial ordering where any monomial involving one of the $S,T,s,s_0,\dots,s_n$ is greater than any monomial in $k[y_0,\dots,y_n,z_0\dots,z_n]$}
\STATE{$h \gets$ the remainder of dividing $fs$ by $G$}
\IF{$h \in k[y_0,\dots,y_n,z_0,\dots,z_n]$}
	\RETURN{$h(1, x_1,\dots,x_{n}, 1,x_1^{-1},\dots,x_{n-1}^{-1},x_n^{-1})$}
\ELSE
	\RETURN \FALSE
\ENDIF
\end{algorithmic}
\end{algorithm}

\begin{lemma} 
Let $\overline{C}$ be a rational normal curve with parametrization $\psi \colon \mathbb{P}^1 \xhookrightarrow{} \mathbb{P}^n$ given by $f_0,\dots,f_n$. \Cref{subalgmembership} correctly returns the pushforward $ \gamma$ of a rational function $\frac{f(S,T)}{g(S,T)}$ on $\psi^{-1}(C)$ along the map given by $\psi^{-1}(C) \xhookrightarrow{} C$, if such a $\gamma$ exists and is regular.
\end{lemma}
\begin{proof}
There exists $\overline{\gamma} \in k[x_0^{\pm 1},\dots,x_n^{\pm 1}]$ such that 
\[
\frac{f}{g} = \overline{\gamma} \big (f_0,\dots,f_n \big )
\]
if and only if there exists $\chi \in k[y_0,\dots,y_n,z_0,\dots,z_n]$ such that 
\[
\frac{f}{g} = \chi \big (f_0,\dots,f_n, f_0^{-1},\dots,f_n^{-1}\big ).
\]
Setting the $s_i$ to be the inverse of the $f_i$ and setting $s$ to be the inverse of $g$, this is equivalent to the statement that $fs$ is in the $k$-algebra generated by $\{f_0,\dots,f_n,s_0,\dots,s_n\}$ in the quotient ring
\[
k[y_0,\dots,y_n,z_0,\dots,z_n,s_0,\dots,s_n,s,S,T]/(f_0s_0 - 1,\dots,f_ns_n - 1, gs-1).
\]
By \protect{\cite[7.3.7]{coxsheaolittle}}, the previous statement is true if and only if $h$, the remainder upon dividing $fs$ by the Gr\"obner basis $G$, is in the polynomial ring $k[y_1,\dots,y_n,z_1,\dots,z_n]$. Suppose $\gamma$ exists, and let $\overline{\gamma}$ be its homogenization. By \protect{\cite[7.3.7]{coxsheaolittle}}, $fs = h(f_0,\dots,f_n,s_0,\dots,s_n)$ and $\overline{\gamma} = h(x_0,\dots,x_n,x_0^{-1},\dots,x_n^{-1})$. Dehomogenizing, we get $\gamma = h(1,x_1,\dots,x_{n}, 1, x_1^{-1},\dots,x_{n}^{-1})$ as the pushforward of $f/g$. Because $h$ is a Laurent polynomial, $\gamma$ is regular on $C$.
\end{proof}

\begin{algorithm}[Computing unit groups of rational normal curves]
\item[]
\label{rationalcurvealg}
\begin{algorithmic}[1]
\REQUIRE{A rational normal curve $\overline{C}$ given parametrically by $f_0(T,S),\dots,f_n(T,S) \in k[S,T]$ and a corresponding very affine curve given by setting $f_0 = 1$}
\ENSURE{A basis of $R^*/k^*$}
\STATE{$D \gets \emptyset$}
\STATE{$[a_1 \colon b_1],\dots,[a_m \colon b_m]\leftarrow$ preimages of $\partial C$ under the parametrization map $\mathbb{P}^1 \xhookrightarrow{} \mathbb{P}^n$.}
\STATE{Choose any basis of $\Div_{\partial}^0(\overline{C})$}
\FORALL{basis elements $\sum_{i}c_i[a_{k_i} \colon b_{k_i}] - \sum_j d_j [a_{l_j} \colon b_{l_j}]$}
	\STATE{$f \leftarrow \prod_i(b_{k_i}S - a_{k_i}T)^{c_i}$}
    \STATE{$g \leftarrow \prod_j(b_{l_j}S - a_{l_j}T)^{d_j}$}
	\STATE{$\overline{\gamma} \leftarrow \Cref{subalgmembership}(f,g,f_0,\dots,f_n)$}
    \STATE{$\gamma \leftarrow \overline{\gamma}(1,x_1,\dots,x_n)$}
    \STATE{$D\leftarrow D\cup \{\gamma\}$}
\ENDFOR
\RETURN{$D$}
\end{algorithmic}
\end{algorithm}

\rationalnormalcurvesoln*
\begin{proof}
Let $C$ be parametrized by polynomials $f_0(S,T),\dots,f_n(S,T) \in k[S,T]$. As $\overline{C} \cong \mathbb{P}^1$, $\Cl_{\partial}^0(\overline{C}) = 0$, so the injection $R^*/k^* \hookrightarrow \Div_{\partial}^0(\overline{C})$ is an isomorphism. For each basis element $\sum_{i}c_i[a_{k_i} \colon b_{k_i}] - \sum_j d_j [a_{l_j} \colon b_{l_j}]$, \Cref{subalgmembership} will produce a rational function $\overline{\gamma}$ on the projective curve which has zeros of order $c_i$ at the points $[f_0(a_{k_i}, b_{k_i}):\dots,f_n(a_{k_i}: b_{k_i})]$ and poles of order $d_j$ at the points $[f_0(a_{l_j}: b_{l_j}),\dots:f_n(a_{l_j}, b_{l_j})]$. By dehomogenizing to arrive at $\gamma$, we get exactly the element of $R^*$ corresponding to our divisor.
\end{proof}

\begin{exam}
Consider the degree 3 rational normal curve $\overline{C} \subseteq \mathbb{P}^3$ given by the parametrization
\[
[S^3 - 4ST^2: S^2T - 9T^3: (S-3T)T^2: (S+3T)T^2]
\]
We compute the following boundary points:
\begin{enumerate}
\item $P_1 = [0 \colon 1]$
\item $P_2 = [1 \colon 0]$
\item $P_3 = [3:1]$
\item $P_4 = [-3:1]$
\item $P_5 = [2:1]$
\item $P_6 = [-2:1]$
\end{enumerate}

We choose the following basis of $\Div_{\partial}^0(\overline{C})$:
\[
P_1 - 2P_2 -P_4+P_5+P_6, P_2 - P_3, P_3 - P_4, P_4 - P_5, P_5 - P_6
\]

Choose coordinates $x, y, z, w$ on $\mathbb{P}^3$. We run \Cref{subalgmembership} to obtain preimages under $\overline{\phi}$ of our basis of $\Div_{\partial}^0(\overline{C})$ in $\Frac(\overline{R})^*$. Their corresponding dehomogenizations with respect to $w$ give a basis of $R^*/k^*$:

\begin{enumerate}
\item $x \rightsquigarrow x$
\item $y \rightsquigarrow y$
\item $z \rightsquigarrow z$
\item $\dfrac{x + 5y + \frac{45}{6}(w - z) + 10(w + z)}{x} \rightsquigarrow \dfrac{x + 5y + \frac{45}{6}(1 - z) + 10(1 + z)}{x}$
\item $\dfrac{x - 4y + 6(w - z) + 4(w+z)}{x} \rightsquigarrow \dfrac{x - 4y + 6(1 - z) + 4(1+z)}{x}$
\end{enumerate}
\end{exam}

\begin{remark}
Although we do not do so here, one could consider various generalizations of the results presented thus far. For example, one can essentially perform the same procedure with ``pinched'' rational curves, i.e. smooth rational curves of degree $> n$ in $\mathbb{P}^n$. However, once higher-dimensional varieties or curves with singularities are considered, the situation becomes more complicated; even computing the boundary is no longer a simple task.
\end{remark}

\section{Elliptic Curves}
\label{ellipticcurves}
Fix $k = \overline{\mathbb Q}$, let $\overline{E} \subseteq \mathbb{P}_k^2$ be an elliptic curve with a given base point $O$, and set $E \coloneqq \overline{E} \cap \mathbb{T}^2$. Due to \Cref{mainexactsequence}, computing the image of $R^*/k^*$ in $\Div_{\partial}^0(\overline{E})$ is equivalent to computing the relations between the closed points of $\partial E =: \{P_1,\ldots, P_n\}$ in $\Cl_{\partial}^0(\overline{E})$. As the group law on the elliptic curve coincides with the group law in the class group, it suffices to compute relations between the corresponding points on the elliptic curve, which can be done via canonical N\'eron--Tate heights.

\subsection{The Canonical N\'eron--Tate Height Pairing}

We briefly define canonical N\'eron--Tate heights, following the exposition from \protect{\cite{aec}}. Speaking broadly, height functions measure the ``arithmetic complexity'' of points on abelian varieties. For any field $F$ and variety $X$, let $X(F)$ denote the $F$-rational points of $X$.

\begin{theorem}[N\'eron--Tate]
\label{nerontatetheorem}
Let $\overline{E}$ be an elliptic curve defined over a number field. There exists a function $\hat{h} \colon \overline{E}(\overline{\mathbb{Q}}) \rightarrow \mathbb{R}$ called the \emph{canonical N\'eron--Tate height} satisfying the following properties:
\begin{enumerate}
    \item For all $P, Q \in \overline{E}(\overline{\mathbb{Q}})$, the parallelogram law holds, i.e.
    \[
        \hat{h}(P + Q) + \hat{h}(P - Q) = 2\hat{h}(P) + 2\hat{h}(Q).
    \]
    \item For all $P \in E(\overline{\mathbb{Q}})$ and $m \in \mathbb{Z}$,
    \[
        \hat{h}(mP) = m^2\hat{h}(P).
    \]
    \item $\hat{h}$ is an even function, and the pairing
    \[
        \langle \quad, \quad \rangle \colon E(\overline{\mathbb{Q}}) \times E(\overline{\mathbb{Q}}) \rightarrow \mathbb{R}
    \]
    \[
        \langle P, Q \rangle = \hat{h}(P+Q) - \hat{h}(P) - \hat{h}(Q)
    \]
    is bilinear. This is equivalent to saying that $\hat{h}$ is a quadratic form on $E(\overline{\mathbb{Q}})$. We call this the canonical N\'eron--Tate height pairing.
    \item For all $P \in \overline{E}(\overline{\mathbb{Q}})$, one has $\hat{h}(P) \geq 0$, and $\hat{h}(P) = 0$ if and only if $P$ is torsion.
\end{enumerate}
\end{theorem}

For any number field $K$, we can obtain a bilinear form on $\overline{E}({K})$ by restricting the bilinear form on $\overline{E}(\overline{\mathbb{Q}})$ in \Cref{nerontatetheorem}(3). This can be extended to a bilinear form on the finite-dimensional real vector space $\overline{E}(K) \otimes \mathbb{R}$.

\begin{prop}[\protect{\cite[VIII.9.9.6]{aec}}]
The N\'eron--Tate height induces a positive definite inner product on $\overline{E}(K) \otimes \mathbb{R}$.
\end{prop}

One can compute heights on elliptic curves efficiently with Algorithm 6.1 in \protect{\cite{heightalg}}.

\subsection{Computing Generators of the Unit Group}

We now detail algorithms to solve Questions \ref{questionone} and \ref{questiontwo} for elliptic curves. First we treat Question \ref{questionone}. In addition to the above theory on N\'eron--Tate heights, we will need the following theorem and subroutines.
\begin{theorem}[\protect{\cite{mahler}}, \protect{\cite[Theorem~4]{weyl}}]
\label{weylthm}
Suppose $L$ is a sublattice in $\mathbb{Z}^n$ of rank $m$. Fix some topological vector space norm on $\mathbb{R}^n$. For all $1 \leq k \leq m$, let $M_k$ denote the minimum size ball centered at the origin that contains $k$ linearly independent vectors in $L$. Then there exists a basis $\{x_1,\dots,x_n\}$ of $L$ such that for all $1 \leq k \leq m$, $|x_k| \leq (\frac{3}{2})^{k-1}M_k$.
\end{theorem}

\begin{subroutine}
\label{torsionrelationssubroutine}
\begin{algorithmic}[1]
\REQUIRE{A set of torsion points $T_1,\dots,T_r$ on an elliptic curve and torsion orders $m_1,\dots,m_n$}
\ENSURE{Generators for the lattice of relations among $T_1,\ldots, T_r$ in $\mathbb{Z}^r$}
\STATE{$D \gets \emptyset$}
\FORALL{$(n_1,\dots,n_r)$ where $0 \leq n_i \leq m_i$}
	\IF{$n_1 T_1 + \dots + n_r T_r = 0$} 
		\STATE {add $(n_1,\dots,n_r)$ to $D$} 
    \ENDIF
\ENDFOR
\RETURN{$D$}
\end{algorithmic}
\end{subroutine}
\Cref{torsionrelationssubroutine} correctly computes all relations among a set of torsion points, as it simply manually checks all possible relations.

\begin{subroutine}
\label{nontorsionrelationssubroutine}
\begin{algorithmic}[1]
\REQUIRE{A set of torsion-free points $Q_1,\dots,Q_n$ on an elliptic curve}
\ENSURE{Generators in $\mathbb{Z}^n$ for the lattice of relations among the $Q_i$ in $\overline{E}(\overline{\mathbb{Q}})/\tors$}
\STATE{Compute the $n \times n$ matrix $A$ such that the $A_{i, j} \gets \langle Q_i,Q_j \rangle = \hat{h}(Q_i + Q_j) - \hat{h}(Q_i) - \hat{h}(Q_j)$}
\RETURN{generators of $\ker A \cap \mathbb{Z}^n$}
\end{algorithmic}
\end{subroutine}

\begin{lemma}
\Cref{nontorsionrelationssubroutine} correctly computes the lattice of relations among the nontorsion points $Q_1, \dots, Q_n$ in $\overline{E}(\overline{\mathbb{Q}})/\tors$.
\end{lemma}
\begin{proof}
Choose some number field $K$ large enough such that $\{Q_1,\dots,Q_n\} \subseteq \overline{E}(K)$. Note that $\overline{E}(K)$ modulo torsion embeds into $\overline{E}(K) \otimes \mathbb{R}$. By \Cref{nerontatetheorem} (3), $A$ is the inner product matrix of a nondegenerate inner product, and thus $\ker A \cap \mathbb{Z}^n$ comprises the relations among the $Q_i$ up to torsion.
\end{proof}

We are now ready to solve \Cref{questionone} for elliptic curves.

\begin{algorithm}[Answering Question 1 for elliptic curves]
\item[]
\label{ellcurvep1}
\begin{algorithmic}[1]
\REQUIRE{An elliptic curve over $\overline{\mathbb{Q}}$ with a nonempty finite set of distinguished points $S \subseteq \overline{E}(\overline{\mathbb{Q}})$ and a base point $O$}
\ENSURE{A minimal generating set of $\ker(\Div_S^0(\overline{E}) \rightarrow \Cl^0(\overline{E}))$}
\STATE{Determine the torsion points of $S$ using heights. Let $Q_1,\dots,Q_n$ refer to torsion-free points, and let $T_1,\dots,T_r$ refer to torsion points with orders $m_1, \dots, m_r$ respectively.}
\STATE{$D \gets \emptyset \subseteq \mathbb{Z}^{n+r}$}
\STATE{$G \gets$ finite subgroup generated by $ (T_1,\dots,T_r) \subseteq \overline{E}(\overline{\mathbb{Q}})$}
\STATE{$D_T \gets$ relations between $T_1,\dots,T_r$ as given by \Cref{torsionrelationssubroutine}}
\FORALL{$(n_1,\dots,n_r)\in D_T$}
    \STATE {add $(0,\dots,0,n_1,\dots,n_r)$ to $D$}
\ENDFOR
\STATE{$D_Q \gets$ relations modulo torsion between $Q_1,\dots,Q_n$ as given by \Cref{nontorsionrelationssubroutine}}
\STATE{$\ell \gets \rank(span_{\mathbb{Z}}(D_Q))$}
\STATE{$\lambda \gets 0$, $S_{\lambda} \gets \emptyset \subset \ZZ^n$}
\WHILE{$\rank(span_{\ZZ}(S_{\lambda})) \neq \ell$} 
	\STATE{$\lambda \gets \lambda + 1$}
    \STATE{$S_{\lambda} \gets \{(m_1,\dots,m_n) \in span_{\mathbb{Z}}(D_Q) \mid \sqrt{\sum m_i^2} \leq \lambda \quad \textrm{and} \quad m_1 Q_1 + \dots + m_n Q_n \in G\}$}
\ENDWHILE
\STATE{$\Lambda \gets \{(m_1,\dots,m_n) \in span_{\mathbb{Z}}(D_Q) \mid \sqrt{\sum m_i^2} \leq {(\frac{3}{2})^{k-1}\lambda} \quad \textrm{and} \quad m_1 Q_1 + \dots + m_n Q_n \in G\}$}
\FOR{$(m_1,\dots,m_n) \in \Lambda$} 
	\STATE {Choose $(n_1,\dots,n_r)$ such that $m_1 Q_1 + \dots + m_n Q_n + n_1 T_1 + \dots + n_r T_r = 0$}
    \STATE {add $(m_1,...,m_n,n_1,\dots,n_r)$ to $D$}
\ENDFOR
\STATE{$L \gets \{(m_1,\dots,m_n,n_1,\dots,n_r) \in \mathbb{Z}^{n+r} \mid \sum m_i + \sum n_j = 0\}$}
\RETURN{a minimal set of generators for $span_{\mathbb{Z}}(D) \cap L$}
\end{algorithmic}
\end{algorithm}

\begin{lemma}
\label{weierstraussq1}
For a distinguished set $S$ of $\overline{\mathbb{Q}}$-points on the elliptic curve $E$, \Cref{ellcurvep1} correctly computes a minimal generating set of the kernel of the map $\Div_S^0(\overline{E}) \rightarrow \Cl^0(\overline{E})$.

\end{lemma}
\begin{proof}
We first prove that the algorithm terminates.  Let $\psi$ denote the map $\Div_{S}(\overline{E}) \rightarrow \Cl_{S}(\overline{E})$, and let $\psi_0$ denote the restriction $\Div_{S}^0(\overline{E}) \rightarrow \Cl_{S}^0(\overline{E})$. Identify $\Div_{S}^0(\overline{E}) \cong \ZZ\langle Q_1,\dots,Q_n,T_1,\dots,T_r\rangle$ with $\mathbb{Z}^{n+r}$ using this ordering of elements in $S$. For any subset $M \subseteq \{1,\dots,n+r\}$, let $\pi_{M}$ denote the projection onto those coordinates.

Note that $\pi_{\{1,\dots,n\}}(\ker \psi) \subseteq span_{\mathbb{Z}}(D_Q)$. In fact $\pi_{\{1,\dots,n\}}(\ker \psi)$ has the same rank as $span_{\ZZ}(D_Q)$; if $(m_1,\dots,m_n) \in span_{\mathbb{Z}}(D_Q)$, then $m_1 Q_1 + \dots +m_n Q_n \in G$ and thus is torsion. It follows that there exists some $m \in \mathbb{Z}$ such that $mm_1 Q_1 + \dots + mm_n Q_n = 0$, so that $(mm_1,\dots,mm_n) \in \pi_{\{1,\dots,n\}}(\ker \psi)$. Hence there exists a $\lambda$ large enough to exit the while loop, and the algorithm terminates. 

We now show the correctness of the algorithm. We claim that $\pi_{\{1,\dots,n\}}(\ker \psi) = span_{\mathbb{Z}}(\Lambda)$. Note by definition that $\Lambda \subseteq \pi_{\{1,\dots,n\}}(\ker \psi)$ so $span_{\mathbb{Z}}(\Lambda) \subseteq \pi_{\{1,\dots,n\}}(\ker \psi)$. By \Cref{weylthm}, as $S_\lambda$ contains at least $\ell$ linearly independent elements, $\Lambda$ will contain a lattice basis of $\pi_{\{1,\dots,n\}}(\ker \psi)$. Thus $\pi_{\{1,\dots,n\}}(\ker \psi) = span_{\mathbb{Z}}(S)$.

Next we show that $span(D) = \ker \psi$. Clearly $span(D) \subseteq \ker\psi$ by construction. Suppose $(m_1,\dots,m_n,n_1,\dots,n_r) \in \ker \psi$. Then $(m_1,\dots,m_n) \in \pi_{\{1,\dots,n\}}(\ker \psi)= span_{\mathbb{Z}}(\Lambda)$, so there exist $n_1',\dots, n_r'$ such that $(m_1,\dots,m_n,n_1',\dots,n_r') \in span(D) \subseteq \ker \psi$. Thus, $(0,\dots,0,n_1 - n_1',\dots,n_r - n_r') \in \ker \psi$. However, $(0,\dots,0,n_1 - n_1',\dots,n_r - n_r') \in span(D)$ because of \Cref{torsionrelationssubroutine}, so 
\[
(m_1,\dots,m_n,n_1',\dots,n_r') + (0,\dots,0,n_1 - n_1',\dots,n_r - n_r') = (m_1,\dots,m_n,n_1,\dots,n_r) \in span(D).
\]
\noindent To conclude, we note that $\ker \psi_0 = \ker \psi \cap L = span_{\mathbb Z}(D)\cap L$.
\end{proof}

We now turn our attention to answering Question \ref{questiontwo}. The following is an explicit version of Miller's algorithm, specialized to genus $1$ \protect{\cite{miller}}.

\begin{algorithm}[Answering Question 2 for elliptic curves]
\item[]
\label{ellcurvep2}
\begin{algorithmic}[1]
\REQUIRE{An elliptic curve $\overline E$ with basepoint $O$ and a divisor $D \in \Div^0(\overline{E})$}
\ENSURE{Whether $D$ is in the image $R^*/k^*$, and an element of $\Frac(\overline{R})^*$ mapping to $D$ if it is}
\STATE{$f \gets 1$}
\WHILE{$|D| \neq 0$}
\IF{$\exists P, Q$ such that $n_P, n_Q > 0$ \AND $P \neq -Q$}
	\STATE{$D \gets D - (P + Q + (-P+Q) - 3O)$}
    \STATE{$f \gets fL$ where $L$ is the line through $P$ and $Q$}
\ELSIF{$\exists P, Q$ such that $n_P, n_Q > 0$ \AND $P = -Q$}
	\STATE{$D \gets D - (P + Q - 2O)$}
    \STATE{$f \gets fL$ where $L$ is the line through $P$ and $Q$}
\ELSIF{$\exists P, Q$ such that $n_P, n_Q < 0$ \AND $P \neq -Q$}
	\STATE{$D \gets D + (P + Q + (-P+Q) - 3O)$}
    \STATE{$f \gets f/L$ where $L$ is the line through $P$ and $Q$}
\ELSIF{$\exists P, Q$ such that $n_P, n_Q < 0$ \AND $P = -Q$}
	\STATE{$D \gets D + (P + Q - 2O)$}
    \STATE{$f \gets f/L$ where $L$ is the line through $P$ and $Q$}
\ELSIF{$D$ is of the form $mP - nQ + oO$ for $m, n \geq 0$}
	\IF{$m \geq 2$ \AND $P$ has order $3$}
    	\STATE{$D \gets D - (3P - 3O)$}
    	\STATE{$f \gets fL$ where $L$ is the tangent line at $P$}
    \ELSIF{$m \geq 2$ \AND $P$ does not have order $2$}
    	\STATE{$D \gets D - (2P + (-2P) - 3O)$}
    	\STATE{$f \gets fL$ where $L$ is the tangent line at $P$}
    \ELSIF{$m \geq 2$ \AND $P$ has order $2$}
    	\STATE{$D \gets D - (2P - 2O)$}
    	\STATE{$f \gets fL$ where $L$ is the tangent line at $P$}
   	\ELSIF{$n \geq 2$ \AND $Q$ has order $3$}
    	\STATE{$D \gets D + (3Q - 3O)$}
    	\STATE{$f \gets f/L$ where $L$ is the tangent line at $Q$}
    \ELSIF{$n \geq 2$ \AND $Q$ does not have order $2$}
    	\STATE{$D \gets D + (2Q + (-2Q) - 3O)$}
    	\STATE{$f \gets f/L$ where $L$ is the tangent line at $Q$}
    \ELSIF{$n \geq 2$ \AND $Q$ has order $2$}
    	\STATE{$D \gets D - (2Q - 2O)$}
    	\STATE{$f \gets f/L$ where $L$ is the tangent line at $Q$}
    \ELSIF{$m = 1$ \AND $n=1$}
    	\STATE{$D \gets D + (Q + (-Q) - 2O)$}
    	\STATE{$f \gets f/L$ where $L$ is the line through $Q$ and $-Q$}
    \ELSIF{($m=1$ \AND $n=0$) \OR ($m=0$ \AND $n=1$)}
    	\RETURN{this divisor is not in the image of $R^*/k^*$}
    \ENDIF
\ENDIF
\ENDWHILE
\RETURN{$f$}
\end{algorithmic}
\end{algorithm}

\begin{lemma}
\label{weierstraussq2}
\Cref{ellcurvep2} correctly determines whether a divisor $D$ is in the image of $R^*/k^*$ and computes an element $f$ of $Frac(\overline R)^*$ mapping to $D$ if so.
\end{lemma}
\begin{proof}
Let $\phi$ denote the map $R^*/k^* \xhookrightarrow{} \Div_0(\overline{E})$. Note that the quantity $D - \phi(f)$ is a loop invariant. Note additionally that during every execution of the loop, exactly one of the conditionals is satisfied; if line 3, 6, 9, and 12 are not satisfied, then $D$ must be of the form $mP - nQ + oO$ for $m,n \geq 0$. If $D = mP - nQ + oO$ then exactly one of line 16, 19, 22, 25, 28, 31, 34, or 37 must be satisfied. $|D|$ is strictly reduced during each iteration unless line 34 or line 37 are satisfied. Line 37 terminates the program. Line 34 cannot be satisfied in two consecutive loops. Thus the algorithm will terminate.

Assume $D \in \phi(R^*/k^*)$. If at some point in execution $|D| = 0$, then as $D$ is degree $0$, $D = 0$ and so $D = \phi(f)$. If, during the execution of the algorithm, line 37 is satisfied, then some point is linearly equivalent to the origin, which is a contradiction. Hence the algorithm outputs an element $f$ with the desired property.

Now assume $D \notin \phi(R^*/k^*)$. Because of the loop invariant $D - \phi(f)$, we will never have $|D| = 0$. Because the algorithm terminates, it must terminate at line 37, as desired.
\end{proof}

Putting together Lemmas \ref{weierstraussq1} and \ref{weierstraussq2} and \ref{questionthreelemma}, we obtain the following result:

\begin{algorithm}[Computing unit groups of elliptic curves]
\label{ellcurvealg}
\item[]
\begin{algorithmic}
\REQUIRE{An elliptic curve over $\overline{\mathbb{Q}}$ with 
boundary points $\partial$
%a nonempty finite set of distinguished points $S \subseteq \overline{E}(\overline{\mathbb{Q}})$ 
and a base point $O$}
\ENSURE{A basis of $R^*/k^*$}
\STATE $V \gets$ a generating set of the image of $R^*/k^*\hookrightarrow \Div_{\partial}^0(\overline E)$ by \Cref{ellcurvep1}
\STATE $B \gets \emptyset$
\FORALL{$v\in V$}
	\STATE $f/g \gets$ rational function with divisor $v$ using \Cref{ellcurvep2}
    \STATE $B\gets B\cup \{h\}$, where $h$ is a Laurent polynomial with the same divisor as $f/g$ by \Cref{computingRepOfUnit}
\ENDFOR
\RETURN{$B$}
\end{algorithmic}
\end{algorithm}

\ellipticcurveqonesoln*

\begin{remark}
Many of the algorithms presented in this section are most easily implemented (e.g. in Sage \cite{sagemath}) for elliptic curves in Weierstrass form. Given a projective isomorphism of $\overline E$ to a Weierstrass form $\overline{W}$ as $\varphi \colon \overline{E} \rightarrow \overline{W}$, we can compute relations among the points in $\varphi(\partial E)$ using \Cref{ellcurvep1}. These relations can be pulled back by $\varphi^{-1}$ to all relations among the points in $\partial E$, because $\varphi$ induces an isomorphism $\Div_{\partial}^0(\overline{E}) \cong \Div_S^0(\overline{W})$. 
%We can then proceed normally.
\end{remark}

\begin{exam}
\label{ellcurveexample}
Let $E$ be the very affine elliptic curve $E = \Spec(\overline{\mathbb{Q}_5}[x^{\pm 1}, y^{\pm 1}]/(y^2 - (x-1)(x+1)(x-4))$ with basepoint $[0:1:0]$. We compute the following six boundary points of $\overline{E} \subseteq \mathbb{P}^2$:

\begin{enumerate}
\item $Q_1 \coloneqq [0: 2:1]$
\item $Q_2 \coloneqq [0: -2:1]$
\item $T_1 \coloneqq [0:1:0]$
\item $T_2 \coloneqq [1: 0:1]$
\item $T_3 \coloneqq [-1: 0:1]$
\item $T_4 \coloneqq [4: 0:1]$
\end{enumerate}

$T_1$ is the identity on $E$ and has torsion order 1; $T_2, T_3,$ and $T_4$ have torsion order 2; and $Q_1$ and $Q_2$ are nontorsion. Algorithm \ref{ellcurvealg} yields the following generating set for the lattice of relations, with corresponding units:

\begin{enumerate}
\item $[1, 1, 1, 1, -2, -2] \rightsquigarrow -y/x^2$
\item $[0, 2, 0, 0, -1, -1] \rightsquigarrow (x-1)/x$
\item $[0, 0, 2, 0, -1, -1] \rightsquigarrow (x+1)/x$
\item $[0, 0, 0, 2, -1, -1] \rightsquigarrow (x-4)/x$
\end{enumerate}
\begin{comment}
\begin{enumerate}
\item $-y/x^2$
\item $(x -1)/x$
\item $(x+1)/x$
\item $(x - 4)/x$
\end{enumerate}
\end{comment}
Note that it is easy to find these units by inspection, but to check that these form a basis of $R^*/k^*$, we rely on the algorithms given in this section.

We can easily compute the tropicalization of the elliptic curve $E \subseteq \mathbb{T}^2$ defined by the equation $y^2 = (x-1)(x+1)(x-4)$ to be three rays emerging from the origin, as seen in Figure \ref{bad-tropicalization-of-elliptic}:
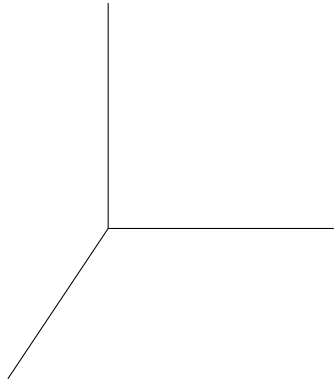
\begin{figure}[h!]
\begin{center}
\begin{tikzpicture}

\draw (3,3) -- (1.66666666666,1);
\draw (6,3) -- (3,3);
\draw (3,6) -- (3,3);
\end{tikzpicture}
\end{center}
\caption{The tropicalization of the elliptic curve in Example \ref{ellcurveexample}.}
\label{bad-tropicalization-of-elliptic}
\end{figure}

Using the unit group basis $\{x, y, x-1, x+1\}$, we compute the intrinsic tropicalization of $E$ in $\mathbb{T}^4$ with Singular, shown in Figure \ref{intrinsic-tropicalization-of-elliptic}:

\begin{figure}[h!]
\begin{center}
\begin{tikzpicture}

\draw (3,3) -- (1,1);
\draw (6,3) -- (3,3);
\draw (3,6) -- (3,3);
\draw (3,3) -- (6,4.5);
\draw (1,1) -- (-2, 1);
\draw (1,1) -- (1, -2);
\end{tikzpicture}
\end{center}
\caption{The intrinsic tropicalization of the elliptic curve in \Cref{ellcurveexample}.}
\label{intrinsic-tropicalization-of-elliptic}
\end{figure}
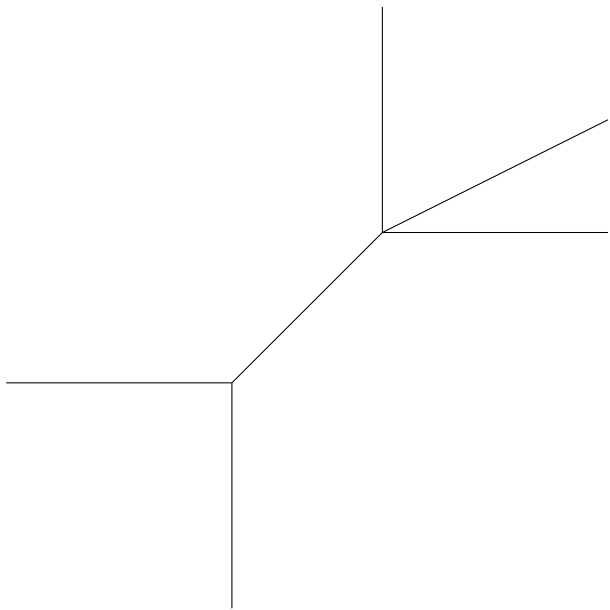

In particular, the intrinsic tropicalization is larger than the original. On the other hand, the $j$-invariant can be computed to be $j(E) = 438976/225$, so that its $5$-adic valuation is -2. It follows from Chan and Sturmfels \cite{chansturmfels} that $\overline E$ can be projectively re-embedded so that its tropicalization is in honeycomb form. The intrinsic tropicalization of $E$ does not retain this information, as this projective re-embedding of $\overline E$ does not preserve our dehomogenization procedure. In particular, intrinsic tropicalizations need not be faithful.
\end{exam}

\subsection{Hyperelliptic curves}
The N\'eron--Tate canonical height can more generally be defined on any abelian variety defined over any number field. Fix some curve $X$ defined over $\overline{\mathbb{Q}}$. Choose a number field $K$ such that $\partial X \subseteq X(K)$. Letting $J$ denote the Jacobian of $X$, the N\'eron--Tate canonical height pairing on $J(\overline{\mathbb{Q}})$ induces a positive definite inner product on $J(K) \otimes \mathbb{R}$.
For curves of genus $2$, Cassels, Flynn, and Smart provide an algorithm to compute the canonical height in \protect{\cite{casselsflynn}} and \protect{\cite{flynnsmart}}, which has since been implemented in Magma. For hyperelliptic curves of genus $3$, Stoll \protect{\cite{stoll}} describes such an algorithm with a corresponding Magma implementation. Additionally, Holmes \protect{\cite{holmes}} has provided a height algorithm for all hyperelliptic curves. Another algorithm to compute heights for all hyperelliptic curves has been provided by M\"{u}ller \protect{\cite{muller}}. However, a hyperelliptic curve of the form $y^2 = f(x)$ in $\mathbb{P}^2$ with $\deg f \geq 4$ has a singularity at infinity, and thus the methods used for elliptic curves do not immediately generalize.
 
\bibliographystyle{amsalpha}
\def\bibfont{\small}
\bibliography{Master.bib}

\providecommand{\bysame}{\leavevmode\hbox to3em{\hrulefill}\thinspace}
\providecommand{\MR}{\relax\ifhmode\unskip\space\fi MR }
\providecommand{\MRhref}[2]{%
  \href{http://www.ams.org/mathscinet-getitem?mr=#1}{#2}
}
\providecommand{\href}[2]{#2}
\begin{thebibliography}{DGPS18}

\bibitem[CF96]{casselsflynn}
J.W.S. {Cassels} and E.V. {Flynn}, \emph{Prolegomena to a middlebrow arithmetic
  of curves of genus 2}, London Mathematical Society Lecture Note Series,
  Cambridge University Press, 1996.

\bibitem[CLO15]{coxsheaolittle}
D.~Cox, J.~Little, and D.~O'Shea, \emph{Ideals, varieties, and algorithms},
  fourth ed., Undergraduate Texts in Mathematics, Springer-Verlag, Switzerland,
  2015.

\bibitem[Coh93]{cohenbook}
H.~Cohen, \emph{A course in computational algebraic number theory}, Graduate
  Texts in Mathematics, Springer-Verlag, 1993.

\bibitem[CS13]{chansturmfels}
M.~Chan and B.~Sturmfels, \emph{{Elliptic Curves in Honeycomb Form, \emph{in}
  Algebraic and Combinatorial Aspects of Tropical Geometry}}, Contemp. Math.
  \textbf{589} (2013), 87--107.

\bibitem[DGPS18]{DGPS}
W.~Decker, G-M Greuel, G.~Pfister, and H.~Sch\"onemann, \emph{{\sc Singular}
  {4-1-1} --- {A} computer algebra system for polynomial computations},
  \url{http://www.singular.uni-kl.de}, 2018.

\bibitem[Eis95]{Eisenbud}
D.~Eisenbud, \emph{Commutative algebra with a view towards algebraic geometry},
  first ed., Graduate Texts in Mathematics, Springer-Verlag, New York, 1995.

\bibitem[FS97]{flynnsmart}
E.V. {Flynn} and N.P. {Smart}, \emph{Canonical heights on the jacobians of
  curves of genus 2 and the infinite descent}, Acta Arith. \textbf{79} (1997),
  333--352.

\bibitem[Fuc60]{Fuchs}
L.~Fuchs, \emph{{Abelian groups}}, International Series of Monographs on Pure
  and Applied Mathematics, Pergamon Press, 1960.

\bibitem[GS]{M2}
Daniel~R. Grayson and Michael~E. Stillman, \emph{Macaulay2, a software system
  for research in algebraic geometry}, Available at
  \url{http://www.math.uiuc.edu/Macaulay2/}.

\bibitem[Hol12]{holmes}
D.~Holmes, \emph{Computing {N\'eron--Tate} heights of points on hyperelliptic
  jacobians}, J. Number Theory \textbf{132} (2012), 1295--1305.

\bibitem[Mah38]{mahler}
K.~Mahler, \emph{{On Minkowski's theory of reduction of positive definite
  quadratic forms}}, Q. J. Math. \textbf{9} (1938), 259--262.

\bibitem[Mil86]{miller}
V.~Miller, \emph{Short programs for functions on curves},
  \url{https://crypto.stanford.edu/miller/miller.pdf}, unpublished.

\bibitem[MS15]{tropicalbook}
D.~{Maclagan} and B.~{Sturmfels}, \emph{Introduction to tropical geometry},
  Graduate Studies in Mathematics, vol. 161, American Mathematical Society,
  2015.

\bibitem[MS16]{heightalg}
J.S. {M{\"u}ller} and M.~{Stoll}, \emph{Computing canonical heights on elliptic
  curves in quasi-linear time}, LMS J. Comput. Math. \textbf{19} (2016),
  391--405.

\bibitem[M{\"u}l13]{muller}
J.S. M{\"u}ller, \emph{Computing canonical heights using arithmetic
  intersection theory}, Math. Comp. \textbf{83} (2013), no.~285, 311--336.

\bibitem[Sam66]{samuel}
P.~Samuel, \emph{A propos du th\'eoreme des unit\'es}, Bulletin des Sciences
  Math\'ematiques \textbf{90} (1966), 89--96.

\bibitem[Sil09]{aec}
J.~H. Silverman, \emph{The arithmetic of elliptic curves}, second ed., Graduate
  Texts in Mathematics, Springer-Verlag, 2009.

\bibitem[{Sto}17]{stoll}
M.~{Stoll}, \emph{{An explicit theory of heights for hyperelliptic Jacobians of
  genus three}}, {Algorithmic and Experimental Methods in Algebra, Geometry,
  and Number Theory} (G.~{B{\"o}ckle}, W.~{Decker}, and G.~{Malle}, eds.),
  Springer-Verlag, 2017, pp.~665--715.

\bibitem[{The}18]{sagemath}
{The Sage Developers}, \emph{{S}agemath, the {S}age {M}athematics {S}oftware
  {S}ystem ({V}ersion 8.2)}, 2018, {\tt http://www.sagemath.org}.

\bibitem[Wey40]{weyl}
H.~Weyl, \emph{Theory of reduction for arithmetical equivalence}, Trans. Amer.
  Math. Soc. \textbf{48} (1940), 126--164.

\end{thebibliography}
\vspace{0.5cm}
\end{document}